\documentclass{amsart}
\usepackage[utf8]{inputenc}
\usepackage{amsmath,amsfonts,amssymb,amsthm}
\usepackage{xcolor}
\usepackage{marginnote}

\newtheorem{theorem}{Theorem}[section]

\newtheorem{proposition}[theorem]{Proposition}
\theoremstyle{definition}

\newtheorem{remark}[theorem]{Remark}

\newcommand{\N}{\mathbb N}

\newcommand{\R}{\mathbb R}

\newcommand{\essinf}{\mathop{\rm essinf\,}}

\newcommand{\Tr}{\texttt{\rmfamily{Tr}}}
\newcommand{\E}{\texttt{\rmfamily{E}}}
\def\div{\mathop{\rm div}}

\begin{document}

\title[Nonlinear equations involving...]{Multiple solutions of nonlinear equations involving the square root of the Laplacian}

\author{Giovanni Molica Bisci}
\address[G. Molica Bisci]{Dipartimento P.A.U., Universit\`a  degli
Studi Mediterranea di Reggio Calabria, Salita Melissari - Feo di
Vito, 89100 Reggio Calabria, Italy} \email{gmolica@unirc.it}

\author{Du\v{s}an Repov\v{s}}
\address[D. Repov\v{s}]{Faculty of Education
 and Faculty of Mathematics and Physics,
  University of Ljubljana, 1000 Ljubljana, Slovenia}
\email{dusan.repovs@guest.arnes.si}

\author{Luca Vilasi}
\address[L. Vilasi]{Department of Mathematical and Computer Sciences, Physical Sciences and Earth Sciences, University of Messina, Viale F. Stagno d'Alcontres 31, 98166 Messina, Italy}
\email{lvilasi@unime.it}



\begin{abstract}
In this paper we examine the existence of multiple solutions of parametric fractional equations involving the square root of the Laplacian $A_{1/2}$ in a smooth bounded domain $\Omega\subset \R^n$ ($n\geq 2$) and with Dirichlet zero-boundary conditions, i.e.
\begin{equation*}
\left\{
\begin{array}{ll}
A_{1/2}u=\lambda f(u) & \mbox{ in } \Omega\\
u=0 & \mbox{ on } \partial\Omega.
\end{array}\right.
\end{equation*}
The existence of at least three $L^{\infty}$-bounded weak solutions is established for certain values of the parameter $\lambda$ requiring that the nonlinear term $f$ is continuous and with a suitable growth. Our approach is based on variational arguments and a variant of Caffarelli-Silvestre's extension method.
\end{abstract}

\maketitle

\tableofcontents

\section{Introduction}\label{sec:introduzione}
In this paper we study under the variational viewpoint the existence of multiple (weak) solutions of the following nonlocal problem
\begin{equation}\tag{$P_\lambda$}\label{problema}
\left\{
\begin{array}{ll}
A_{1/2}u=\lambda \beta(x)f(u) & \mbox{ in } \Omega\\
u=0 & \mbox{ on } \partial\Omega,
\end{array}\right.
\end{equation}
where $\Omega$ is a bounded open subset of $\R^n$ ($n\geq 2$) with Lipschitz boundary $\partial \Omega$, $\lambda$ is a positive real parameter and $\beta:\Omega\to\R$ is a function belonging to $L^{\infty}(\Omega)$ and satisfying
\begin{equation}\label{proprietabeta}
\beta_0:=\essinf_{x\in\Omega}\beta(x) > 0.
\end{equation}

The fractional operator $A_{1/2}$ that appears in \eqref{problema} is defined by using the approach developed in the pioneering works of Caffarelli \& Silvestre \cite{CSCPDE}, Caffarelli \& Vasseur \cite{CafVas}, Cabr\'{e} \& Tan \cite{cabretan}, to which we refer in Section 2 for the precise mathematical description and related properties. As it was pointed out in \cite{cabretan}, the fractions of the Laplacian, like the square root of the Laplacian $A_{1/2}$, are the infinitesimal generators of L\'{e}vy stable diffusion processes and they appear, among the other things, in anomalous diffusions in plasmas, flames propagation, population dynamics, geophysical fluid dynamics and American options in finance.

 Elliptic equations involving fractional powers of the Laplacian have recently been treated in \cite{colorado1, colorado2, MaMu, MRadu2, MuPa, tan} (see also the references therein). In particular, Cabr\'e \& Tan in \cite{cabretan}  have studied the existence, non-existence and regularity of positive solutions of problem \eqref{problema} with power-type nonlinearities, together with \`{a} priori estimates of Gidas-Spruck type and symmetry results of Gidas-Ni-Nirenberg type. Along the same direction we mention the paper \cite{tan3} by Tan.

In the current context, regarding the nonlinear term, we assume that $f:\R\to\R$ is continuous and
\begin{itemize}
\item[] \it{there exist two non-negative constants $a_1,a_2$ and $q\in \left(1, 2n/(n-1)\right)$ such that
\begin{equation}\label{S}
|f(t)|\leq a_1+a_2|t|^{q-1}
\end{equation}
for every $t\in\R$;}
\item[] \it{the potential $F(t):=\displaystyle\int_0^{t} f(\xi)d\xi$ satisfies the sign-condition
\begin{equation}\label{Segno}
\inf_{t\in [0,+\infty)}F(t)\geq 0,
\end{equation}
}
\end{itemize}
in addition to some technical, yet central algebraic assumptions that will be stated later on (see \eqref{AI}).

A special case of our main result, which ensures the multiplicity of positive solutions when $\lambda$ is big enough, reads as follows:

\begin{theorem}\label{boccia}
Let $r>0$ and denote
$$
\Gamma^0_r:=\{(x,0)\in \partial \R^{n+1}_+:|x|<r\},
$$
where $\partial \R^{n+1}_+:=\R^n\times (0,+\infty)$ and $n\geq 2$. Let $f:[0,+\infty)\rightarrow \R$ be a continuous function such that
$$
f(t)\leq a_2t^{m}
$$
for all $t\in [0,+\infty)$ and for some $m\in \left( 1,(n+1)/(n-1)\right) $. Furthermore, assume that there is $\zeta>0$ such that $f(t)>0$ for every $t\in(0,\zeta)$ and $f(\zeta)=0$.

Then there exists $\lambda^{\star}>0$ such that, for every
$\lambda>\lambda^{\star}$,
the following nonlocal problem
\begin{equation}\label{problema24}
\left\{
\begin{array}{ll}
A_{1/2}u= \lambda \beta(x)f(u) & \mbox{\rm in } \Gamma^0_r\\
u> 0 & \mbox{\rm on } \Gamma^0_r\\
u=0 & \mbox{\rm on } \partial\Gamma^0_r,
\end{array}\right.
\end{equation}
admits at least two distinct weak solutions $u_{1,\lambda}, u_{2,\lambda}\in L^{\infty}(\Gamma^0_r) \cap H_0^{1/2}(\Gamma^0_r)$.
\end{theorem}

The above theorem is the fractional analogue, on the $n$-dimensional Euclidean ball, of a classical multiplicity result due to Ambrosetti \& Rabinowitz (see \cite[Theorem 2.32]{rabinowitz}).

In the more general Theorem \ref{generalkernel0f} we determine a precise (bounded) interval of positive parameters $\lambda$ for which the nonlocal problem \eqref{problema} admits at least three $L^\infty$-bounded weak solutions in a suitable fractional Sobolev space $X_0^{1/2}(\Omega)$ (see Section \ref{sec:preliminaries}). Such a result is proved by exporting certain variational techniques to the fractional framework. Indeed, following the paper \cite{cabretan}, we first transform problem \eqref{problema} to a local problem in a cylinder $\mathcal{C}_\Omega$, by using the notion of harmonic extension and the Dirichlet-to-Neumann map on $\Omega$ (cf. Section \ref{sec:preliminaries}). Successively, we study the existence of critical points of the energy functional $\mathcal J_{\lambda}$ associated with the extended problem. A local minimum result for differentiable functionals (\cite{Ricc2}), the classical minimization technique and a critical point result due to Pucci and Serrin (\cite[Theorem 4]{ps}) are the abstract tools behind the existence of such critical points.\par

Indeed, the first weak solution turns out to be a local minimum $w_{1,\lambda}$ for the energy functional $\mathcal J_{\lambda}$ (Proposition \ref{Prop1}); the second one, say $w_{2,\lambda}$, is obtained as a global minimum of $\mathcal J_{\lambda}$ as a byproduct of the sequential weak lower semicontinuity and coercivity of $\mathcal J_{\lambda}$. Finally, thanks to Propositions \ref{Prop2} and \ref{Prop3}, we are able to distinguish them, that is $w_{1,\lambda}\neq w_{2,\lambda}$, and since $\mathcal J_{\lambda}$ fulfills the Palais-Smale condition, Pucci-Serrin's result guarantees the existence of a third critical point $w_{3,\lambda}\notin \{w_{1,\lambda}, w_{2,\lambda}\}$ (see Theorem \ref{generalkernel0f}).\par
 The traces of these solutions give back three weak solutions to \eqref{problema} that are bounded in $L^{\infty}(\Omega)$ owing to \cite[Theorem 5.2]{cabretan}.\par
 
 The paper is structured as follows. In Section 2 we introduce notation, preliminary notions and variational framework. In Section \ref{preparatoryresults} we state and prove our main result with some applications.

\section{Preliminaries}\label{sec:preliminaries}
We start by recalling the functional space setting naturally associated with \eqref{problema}, first introduced in \cite{cabretan}; we also refer to the recent book \cite{MRS}, as well as \cite{valpal}, for detailed accounts on Sobolev spaces of fractional order.

\subsection{Fractional Sobolev spaces}
The power $A_{1/2}$ of the Laplace operator $-\Delta$ in a bounded domain $\Omega$ with zero-boundary conditions is defined via the spectral decomposition, using the powers of the eigenvalues of the original operator. Hence, according to classical results on positive operators in $\Omega$, if $\left\lbrace (\varphi_j,\lambda_j)\right\rbrace _{j\in \N}$ are the eigenfunctions and eigenvalues of the usual linear Dirichlet problem
\begin{equation}\label{problema autovalori}
\left\{\begin{array}{ll}
-\Delta u=\lambda u & \mbox{in } \Omega\\
u=0 & \mbox{on } \partial\Omega,
\end{array}
\right.
\end{equation}
then $\left\lbrace \left( \varphi_j,\lambda_j^{1/2}\right) \right\rbrace _{j\in\N}$ are the eigenfunctions and eigenvalues of the corresponding fractional one:
\begin{equation}\label{problema autovalori2}
\left\{\begin{array}{ll}
A_{1/2} u=\lambda u & \mbox{in } \Omega\\
u=0 & \mbox{on } \partial\Omega.
\end{array}
\right.
\end{equation}

As usual we consider each eigenvalue $\lambda_j$ repeated according to its (finite) multiplicity,
$$
0<\lambda_{1}<\lambda_{2}\le \dots \le \lambda_{j}\le \lambda_{j+1}\le \dots
$$
It is well known that $\lambda_j\rightarrow +\infty$ as $j\rightarrow +\infty$. Moreover, we can suppose the eigenfunctions $\{\varphi_j\}_{j\in \N}$ are normalized as follows:
$$
\int_\Omega |\nabla\varphi_j(x)|^2dx=\lambda_j\int_\Omega |\varphi_j(x)|^2dx=\lambda_j,\quad\forall\, j\in \N,
$$
and
$$
\int_\Omega \nabla\varphi_i(x)\cdot\nabla\varphi_j(x)dx=\int_\Omega \varphi_i(x)\varphi_j(x)dx=0,\quad\forall\, i\neq j.
$$
Finally, standard regularity arguments ensure that $\varphi_j\in C^{2}(\overline \Omega)$ for every $j\in \N$.

The operator $A_{1/2}$ turns out to be well-defined on the Sobolev space
$$
H_0^{1/2}(\Omega):=\left\{u\in L^2(\Omega): u=\displaystyle\sum_{j=1}^{\infty}a_j\varphi_j \mbox{  and } \displaystyle\sum_{j=1}^{\infty} a_j^2\lambda_j^{1/2}<+\infty\right\},
$$
endowed with the norm
$$
\|u\|_{H_0^{1/2}(\Omega)}:=\displaystyle\left(\sum_{j=1}^{\infty} a_j^2\lambda_j^{1/2}\right)^{1/2},
$$
and it has the following form
$$
A_{1/2}u:=\sum_{j=1}^{\infty} a_j\lambda_j^{1/2}\varphi_j, \quad \forall u\in H_0^{1/2}(\Omega).
$$

\subsection{The extension problem}\label{subsec:b}
Associated with the bounded domain $\Omega$, let us consider the cylinder
$$
\mathcal{C}_\Omega:=\{(x,y): x\in \Omega,\,\, y>0\}\subset \R_+^{n+1},
$$
and denote by $\partial_L \mathcal{C}_\Omega:=\partial\Omega\times [0,+\infty)$ its lateral boundary.

For a function $u\in H_0^{1/2}(\Omega)$ define the harmonic extension $\E(u)$ to $\mathcal{C}_\Omega$ as the solution of the problem
\begin{equation}\label{1.1}
 \left\{
\begin{array}{ll}
\div(\nabla \E(u))=  0  &\mbox{ in } \mathcal{C}_\Omega\\
\displaystyle{\E(u)=0}  &\mbox{ on } \partial_L\mathcal{C}_\Omega\\
\Tr(\E(u))=u  &\mbox{ on } \Omega,
\end{array}\right.
\end{equation}
where
$$
\Tr(\E(u))(x):=\E(u)(x,0),\quad\, \forall\, x\in \Omega.
$$

The extension function $\E(u)$ belongs to the Hilbert space
$$
X^{1/2}_0(\mathcal{C}_\Omega):=\left\{w\in L^2(\mathcal{C}_\Omega):w=0\,\,{\rm on }\,\, \partial_L \mathcal{C}_\Omega,\,\, \int_{\mathcal{C}_\Omega}|\nabla w(x,y)|^2\,dxdy<+\infty\right\},
$$
equipped with the standard norm
$$
\|w\|_{X^{1/2}_0(\mathcal{C}_\Omega)}:=\left(\int_{\mathcal{C}_\Omega}|\nabla w(x,y)|^2\,dxdy\right)^{1/2},
$$
which can also be characterized as
$$
X^{1/2}_0(\mathcal{C}_\Omega)=\left\{w\in L^2(\mathcal{C}_\Omega):w=\displaystyle\sum_{j=1}^{\infty}b_j\varphi_je^{-\lambda_j^{1/2}y} \mbox{ with } \displaystyle\sum_{j=1}^{\infty} b_j^2\lambda_j^{1/2}<+\infty\right\},
$$
see \cite[Lemma 2.10]{cabretan}.

In our framework, a crucial relationship between the spaces $X^{1/2}_0(\mathcal{C}_\Omega)$ and $H_0^{1/2}(\Omega)$ introduced above is played by the trace operator $\Tr:X^{1/2}_0(\mathcal{C}_\Omega)\rightarrow H_0^{1/2}(\Omega)$ defined by
$$
\Tr(w)(x):=w(x,0),\,\quad \forall\, x\in \Omega.
$$
which turns out to be a continuous map (see \cite[Lemma 2.6]{cabretan}). We also notice that
$$
H^{1/2}_0(\Omega)=\left\lbrace u\in L^2(\Omega):u=\Tr(w) \mbox{ for some } w\in X^{1/2}_0(\mathcal{C}_\Omega)\right\rbrace \subset H^{1/2}(\Omega),
$$
and that the extension operator $\E:H_0^{1/2}(\Omega)\rightarrow X^{1/2}_0(\mathcal{C}_\Omega)$ is an isometry, i.e.
$$
\|\E(u)\|_{X^{1/2}_0(\mathcal{C}_\Omega)}=\|u\|_{H_0^{1/2}(\Omega)},
$$
for every $u\in H_0^{1/2}(\Omega)$. Here $H^{1/2}(\Omega)$ denotes the Sobolev space of order $1/2$ defined by
$$
H^{1/2}(\Omega):=\left\{u\in L^2(\Omega):\int_{\Omega\times\Omega}\frac{|u(x)-u(y)|^2}{|x-y|^{n+1}}dxdy<+\infty\right\},
$$
with the norm
$$
\|u\|_{H^{1/2}(\Omega)}:=\left(\int_{\Omega\times\Omega}\frac{|u(x)-u(y)|^2}{|x-y|^{n+1}}dxdy+\int_\Omega|u(x)|^2dx\right)^{1/2}.
$$

Clearly, for every $w\in X^{1/2}_0(\mathcal{C}_\Omega)$ the following trace inequality holds
\begin{equation}\label{traceine}
\|\Tr(w)\|_{H^{1/2}_0(\Omega)}\leq \|w\|_{X^{1/2}_0(\mathcal{C}_\Omega)},
\end{equation}
while, due to \cite[Lemmas 2.4 and 2.5]{cabretan},
the embedding $\Tr(X_0^{1/2}(\mathcal{C}_\Omega))\hookrightarrow
L^p(\Omega)$ is continuous for any $p\in [1,2^\sharp]$ and compact whenever $p\in [1,2^\sharp)$, where $2^\sharp:=2n/(n-1)$ denotes the \textit{fractional critical Sobolev exponent}. Thus, if $p\in [1,2^\sharp]$, then there exists a positive constant $c_p$ (depending on $p$, $n$ and $|\Omega|$) such that
\begin{equation}\label{constants}
\left(\displaystyle\int_{\Omega}|\Tr(w)(x)|^p dx\right)^{1/p}\leq c_p \left(\displaystyle\int_{\mathcal{C}_\Omega}|\nabla w(x,y)|^2\,dxdy\right)^{1/2},
\end{equation}
for every $w\in X_0^{1/2}(\mathcal{C}_\Omega)$.

As we have briefly outlined in the introduction, we will adopt the following equivalent definition of the square root of the Laplacian (see for instance \cite{colorado1, colorado2, cabretan}). By using the extension $\E(u)\in X^{1/2}_0(\mathcal{C}_\Omega)$ of a function $u\in H_0^{1/2}(\Omega)$, i.e. the solution of \eqref{1.1}, the fractional operator $A_{1/2}$ in $\Omega$ acting on $u$ agrees with the map
$$
A_{1/2}u(x):=-\lim_{y\rightarrow 0^+}\frac{\partial \E(u)(x,y)}{\partial y}\quad\forall\, x\in \Omega,
$$
i.e.
$$
A_{1/2}u(x)=\displaystyle\frac{\partial \E(u)(x,0)}{\partial \nu}\quad\,\forall\, x\in \Omega
$$
where $\nu$ is the unit outer normal to $\mathcal{C}_{\Omega}$ at $\Omega\times\{0\}$.

\subsection{Weak solutions}
Fix $\lambda>0$ and assume that $f:\R\rightarrow\R$ is continuous and satisfies \eqref{S}. We say that a function $u=\Tr(w)\in H^{1/2}_0(\Omega)$ is a \textit{weak solution} of \eqref{problema} if $w\in X^{1/2}_0(\mathcal{C}_\Omega)$ weakly solves
\begin{equation}\label{2}
\left\{
\begin{array}{ll}
-\div(\nabla w)=  0  & \mbox{ in } \mathcal{C}_\Omega\\
w=0 & \mbox{ on } \partial_L\mathcal{C}_\Omega\\
\displaystyle\frac{\partial w}{\partial \nu}= \lambda\beta(x)f(\Tr(w)) & \mbox{ on } \Omega,
\end{array}\right.
\end{equation}
i.e.
\begin{equation}\label{EL}
\int_{\mathcal{C}_\Omega}\langle \nabla w,\nabla \varphi \rangle dxdy=\lambda\int_{\Omega}\beta(x)f(\Tr (w)(x))\Tr(\varphi)(x)dx,
\end{equation}
for every $\varphi\in X^{1/2}_0(\mathcal{C}_\Omega)$.

As direct computations prove, equation \eqref{EL} represents the variational formulation of \eqref{2} and the energy functional $\mathcal J_{\lambda}:X^{1/2}_0(\mathcal{C}_\Omega)\to \R$ associated with \eqref{EL} is defined by
\begin{equation}\label{JKlambda}
\mathcal J_{\lambda}(w) :=\frac{1}{2}\int_{\mathcal{C}_\Omega}|\nabla w(x,y)|^2\,dxdy -\lambda\int_{\Omega}\beta(x)F(\Tr (w)(x))dx,
\end{equation}
for every $w\in X^{1/2}_0(\mathcal{C}_\Omega)$. Indeed, under our assumptions on the nonlinear term, it is straightforward to show that
$\mathcal J_{\lambda}$ is well-defined and of class $C^1$ in $X^{1/2}_0(\mathcal{C}_\Omega)$ and that its critical points
are exactly the weak solutions of problem \eqref{2}. The traces of such critical points being weak solutions of \eqref{problema}, we can employ methods from critical point theory to attack problem \eqref{problema}.

To this end, the following abstract theorem due to Ricceri (see \cite{Ricc2}), restated here in a more convenient form, plays a key role in our study.

\begin{theorem}\label{BMB}
Let $X$ be a reflexive real Banach space, and let $\Phi,\Psi:X\to\R$ be two G\^{a}teaux differentiable functionals such that $\Phi$ is
strongly continuous, sequentially weakly lower semicontinuous and coercive. Furthermore, assume that $\Psi$ is sequentially weakly upper semicontinuous. For every $r>\inf_X
\Phi$, put
$$
\varphi(r):=\inf_{w\in\Phi^{-1}((-\infty,r))}\frac{\left(\displaystyle\sup_{z\in\Phi^{-1}((-\infty,r))}\Psi(z)\right)-\Psi(w)}{r-\Phi(w)}.
$$
Then for each $r>\inf_X\Phi$ and each
$\lambda\in\left(0,{1}/{\varphi(r)}\right)$, the restriction
of $J_\lambda:=\Phi-\lambda\Psi$ to
$\Phi^{-1}((-\infty,r))$ admits a global minimum, which is a
critical point $($local minimum$)$ of $J_\lambda$ in $X$.
\end{theorem}

In addition to the above result, the following classical theorem by Pucci and Serrin (see \cite[Theorem 4]{ps} and \cite[Theorem 3.10]{rabinowitz}) will allow us to deduce the existence of a further critical point.

\begin{theorem}\label{PucciSerrin}
Let $J:X\rightarrow\R$ be a $C^1$-functional satisfying the $(\rm PS)$ condition. If $J$ has a pair of local minima or maxima, then $J$ admits a third critical point.
\end{theorem}

For the sake of completeness we recall that, if $X$ is a real Banach space, a $C^1$-functional $J:X\to\R$ is said to satisfy the Palais-Smale condition at level $\mu\in\R$ when
\begin{itemize}
\item[$\textrm{(PS)}_{\mu}$] {\it Every sequence $\{z_{j}\}_{j\in \N} \subset X$  such that
$$
J(z_j)\to \mu \quad{\it and}\quad \|J'(z_j)\|_{X^*}\to 0
$$
as $j\rightarrow +\infty$, possesses a convergent subsequence in $X$.}
\end{itemize}

Here $X^*$ denotes the topological dual of $X$. We say that $J$ satisfies the Palais-Smale condition ($(\rm PS)$ for short) if $\textrm{(PS)}_{\mu}$ holds for every $\mu\in \R$.

\section{The main result}\label{preparatoryresults}
Define the functionals $\Phi,\Psi:X^{1/2}_0(\mathcal{C}_\Omega)\to\R$, naturally associated with \eqref{problema}, by
$$
\Phi(w):=\frac{1}{2}\|w\|^2_{X^{1/2}_0(\mathcal{C}_\Omega)}, \quad \Psi(w):=\int_\Omega \beta(x)F(\Tr(w)(x))dx,\quad\forall\, w\in X^{1/2}_0(\mathcal{C}_\Omega).
$$

Clearly, $\Phi$ is a coercive, continuously G\^{a}teaux-differentiable and sequentially weakly lower semicontinuous functional. On the other hand, $\Psi$ is well-defined, continuously G\^{a}teaux-differentiable and, on account of \eqref{S} and the compactness of the embedding $\Tr(X_0^{1/2}(\mathcal{C}_\Omega))\hookrightarrow
L^p(\Omega)$, $p\in [1,2^\sharp)$, also weakly continuous in $X^{1/2}_0(\mathcal{C}_\Omega)$.

It is easy to deduce that
\begin{equation}\label{derivatefunz}
\Phi'(w)(\varphi)=\int_{\mathcal{C}_\Omega}\langle \nabla w,\nabla \varphi \rangle dxdy,\quad { \rm and }\quad \Psi'(w)(\varphi)=\int_{\Omega}\beta(x)f(\Tr (w)(x))\Tr(\varphi)(x)dx,
\end{equation}
for every $\varphi\in X^{1/2}_0(\mathcal{C}_\Omega)$.

The aim of this section is to establish a precise interval of values of the parameter $\lambda$ for which the functional $\mathcal J_{\lambda}$ admits at least three critical points.

To this end, fix a point $x_0\in \Omega$ and choose $\tau>0$ so that
\begin{equation}\label{tiau}
{B}(x_0,\tau):=\{x\in\R^n:|x-x_0|< \tau\}\subseteq\Omega.
\end{equation}
Define also
$$
g_\Omega^{(n)}:=\frac{2^n-1}{2^{n-1}}\tau^{n-2}\omega_n, \quad h_\Omega^{(n)}:= g_\Omega^{(n)} + \frac{|\Omega|}{8},
$$
and
$$
K_1:=\frac{\sqrt{2}c_1 h_{\Omega}^{(n)}\|\beta\|_{\infty}}{\omega_n\displaystyle\beta_0}\left(\frac{2}{\tau}\right)^{n},\quad K_2:=\frac{2^{q/2}c_q^q h_{\Omega}^{(n)}\|\beta\|_{\infty}}{q\omega_n\displaystyle\beta_0}\left(\frac{2}{\tau}\right)^{n},
$$
where
$$
\omega_n:=\frac{\pi^{n/2}}{\displaystyle\Gamma\left(1+\frac{n}{2}\right)}
$$
denotes the Lebesgue measure of the unit ball in $\R^n$ and
$$
\Gamma(t):=\int_0^{+\infty}z^{t-1}e^{-z}dz, \quad\forall t>0,
$$
is the classical gamma function.

\begin{theorem}\label{Prop1}
Let $f:\R\rightarrow\R$ be a function satisfying \eqref{S}. Then for every $\gamma>0$ and every $\lambda<\mu_2$, with
\begin{equation}
\mu_2:=\frac{2^n}{\tau^n\omega_n\beta_0}\left(\frac{h^{(n)}_\Omega\gamma}{a_1 K_1 + a_2 K_2\gamma^{q-1}}\right),
\end{equation}
there exists a local minimum $w_{1,\lambda}\in \Phi^{-1}((-\infty,\gamma^2))$ of $\mathcal J_{\lambda}$ in $X^{1/2}_0(\mathcal{C}_\Omega)$.
\end{theorem}

\begin{proof}
Owing to the growth condition \eqref{S}, one has
\begin{equation}\label{inequality}
F(t)\leq a_1|t| + \frac{a_2}{q}|t|^{q},
\end{equation}
for every $t\in\R$. With the idea of using Theorem \ref{BMB}, let us consider the function
$$
\chi(r):=\frac{\displaystyle\sup_{z\in\Phi^{-1}((-\infty,r])}\Psi(z)}{r},
$$
with $r\in (0,+\infty)$. It follows from \eqref{inequality} and \eqref{constants} that, for each $z\in X^{1/2}_0(\mathcal{C}_\Omega)$,
\begin{align*}
\Psi(z)&=\int_\Omega \beta(x)F(\Tr(z)(x))dx\\
&\leq \left(a_1\|\Tr(z)\|_{L^1(\Omega)}+\frac{a_2}{q}\|\Tr(z)\|_{L^q(\Omega)}^q\right)\|\beta\|_{\infty}\\
& \leq \left(a_1c_1\|z\|_{X^{1/2}_0(\mathcal{C}_\Omega)} + \frac{a_2}{q} c_q^q\|z\|_{X^{1/2}_0(\mathcal{C}_\Omega)}^q\right)\|\beta\|_{\infty}
\end{align*}
and therefore
\begin{eqnarray}\label{hh}
\sup_{z\in\Phi^{-1}((-\infty,r])}\Psi(z)\leq\sqrt{2r} a_1 c_1\|\beta\|_{\infty} + \frac{(2r)^{q/2}a_2 c_q^q}{q}\|\beta\|_{\infty}.
\end{eqnarray}
The above inequality yields
\begin{equation}\label{n}
\chi(r)\leq \sqrt{\frac{2}{r}} a_1 c_1\|\beta\|_{\infty} + \frac{2^{q/2} a_2 c_q^q}{q}r^{q/2-1}\|\beta\|_{\infty}
\end{equation}
for every $r>0$ and therefore
\begin{align*}\label{c}
\chi(\gamma^2)=\frac{\displaystyle\sup_{z\in\Phi^{-1}((-\infty,\gamma^2])}\Psi(z)}{\gamma^2} &\leq\sqrt{2}\frac{a_1c_1}{\gamma}\|\beta\|_{\infty} + \frac{2^{q/2}a_2 c_q^q}{q}\gamma^{q-2}\|\beta\|_{\infty} \\
                   & =\beta_0\frac{\tau^n}{2^n}\frac{\omega_n}{h^{(n)}_\Omega}\left(a_1\displaystyle\frac{K_1}{\gamma}+a_2K_2\gamma^{q-2}\right)\\
                   & =\frac{1}{\mu_2}.
\end{align*}

Now, bearing in mind that $0\in\Phi^{-1}((-\infty,{\gamma}^2))$ and $\Phi(0)=\Psi(0)=0$, we observe that
$$
\varphi({\gamma}^2):=\inf_{w\in\Phi^{-1}((-\infty,{\gamma}^2))}\frac{\displaystyle\left(\sup_{z\in\Phi^{-1}((-\infty,{\gamma}^2))}\Psi(z)\right)-\Psi(w)}{\gamma^2-\Phi(w)}\leq \chi({\gamma}^2)
$$
and thus
$$
\lambda\in (0,\mu_2)\subseteq (0,{1}/{\varphi({\gamma}^2)}).
$$

Then, in the light of Theorem \ref{BMB} there exists a function $w_{1,\lambda}\in\Phi^{-1}((-\infty,{\gamma}^2))$ such that
$$
\mathcal J_{\lambda}'(w_{1,\lambda})=\Phi'(w_{1,\lambda})-\lambda\Psi'(w_{1,\lambda})=0
$$
and, in particular, $w_{1,\lambda}$ is a global minimum of the restriction of $\mathcal J_{\lambda}$ to $\Phi^{-1}((-\infty,{\gamma}^2))$. This completes the proof.
\end{proof}

For the sequel we need to define suitable test functions in the space $X^{1/2}_0(\mathcal{C}_\Omega)$. Take two positive constants $\gamma$ and $\varrho$ such that
\begin{equation}\label{tiau2}
\varrho>\frac{\gamma}{\sqrt{g^{(n)}_\Omega}}
\end{equation}
and define the truncated cones $\omega^{\varrho}_{\tau}:\Omega \to \R$ as follows:
\begin{equation*}\label{deftruncfunc}
\omega^{\varrho}_{\tau}(x):=
\left\{
\begin{array}{ll}
0 & \mbox{ if $x \in \overline\Omega \setminus B(x_0,\tau)$} \\
\displaystyle\frac{2\varrho}{\tau} \left(\tau- |x-x_0| \right)
& \mbox{ if $x \in B(x_0,\tau) \setminus B(x_0,\tau/2)$} \\
\varrho & \mbox{ if $x \in B(x_0,{\tau}/2).$}
\end{array}
\right.
\end{equation*}

It is easily seen that
\begin{equation}\label{norma}
\begin{aligned}
  \int_\Omega |\nabla\omega^{\varrho}_{\tau}(x)|^2\,dx & = \int_{B(x_0,\tau)\setminus B(x_0,\tau/2)}\frac{4\varrho^2}{\tau^2}dx \\
                                & = \frac{4\varrho^2}{\tau^2} (|B(x_0,\tau)|-|B(x_0,\tau/2)|)\\
                                & = 4\varrho^2 \omega_n\tau^{n-2}\left(1-\frac{1}{2^n}\right).
\end{aligned}
\end{equation}
\indent Let
$$
w_{\tau}^{\varrho}(x,y):=e^{-\frac{y}{2}}\omega_\tau^{\varrho}(x),\quad\, \forall (x,y)\in \mathcal{C}_\Omega.
$$
Clearly, $w_{\tau}^{\varrho}\in X_0^{1/2}(\mathcal{C}_\Omega)$ and, since
$$
|\nabla w_{\tau}^{\varrho}(x,y)|^2=e^{-y}|\nabla \omega_{\tau}^{\varrho}(x)|^2+\frac{1}{4}e^{-y}|\omega_{\tau}^{\varrho}(x)|^2,\quad\, \forall (x,y)\in \mathcal{C}_\Omega
$$
it follows that
\begin{equation}\label{norma2}
\begin{aligned}
\|w_{\tau}^{\varrho}\|^2_{X^{1/2}_0(\mathcal{C}_\Omega)} & := \int_{\mathcal{C}_\Omega}|\nabla w_{\tau}^{\varrho}(x,y)|^2\,dxdy\\
                   & = \int_{\mathcal{C}_\Omega}e^{-y}|\nabla \omega^{\varrho}_\tau(x)|^2\,dxdy+\frac{1}{4} \int_{\mathcal{C}_\Omega}e^{-y}|\omega^{\varrho}_\tau(x)|^2\,dxdy\\
                   & =\int_0^{+\infty}e^{-y}dy \left(\int_{\Omega}|\nabla \omega^{\varrho}_\tau(x)|^2\,dx+\frac{1}{4} \int_{\Omega}|\omega^{\varrho}_\tau(x)|^2\,dx\right)\\
                   & =\int_{\Omega}|\nabla \omega^{\varrho}_\tau(x)|^2\,dx+\frac{1}{4} \int_{\Omega}|\omega^{\varrho}_\tau(x)|^2\,dx.\\
\end{aligned}
\end{equation}

Thus, \eqref{norma} and \eqref{norma2} provide the estimate
\begin{equation}\label{norma3}
4 \omega_n\tau^{n-2}\left(1-\frac{1}{2^n}\right)\varrho^2 \leq \|w_{\tau}^{\varrho}\|^2_{X^{1/2}_0(\mathcal{C}_\Omega)}\leq \left(4 \omega_n\tau^{n-2}\left(1-\frac{1}{2^n}\right)+\frac{|\Omega|}{4}\right)\varrho^2.
\end{equation}

Define
$$
\mu_1:=\left(\frac{2^n h^{(n)}_\Omega}{\omega_n\tau^n\beta_0}\right)\frac{\varrho^2}{F(\varrho)}.
$$

\begin{proposition}\label{Prop2}
The following inequality holds
\begin{equation}\label{phigamma}
\Phi(w_{\tau}^{\varrho})>\gamma^2.
\end{equation}
In addition, assuming that $\mu_1<\mu_2$ and $\lambda\in(\mu_1,\mu_2)$, one has
\begin{equation}\label{phigamma2}
\Phi(w_{\tau}^{\varrho})-\lambda\Psi(w_{\tau}^{\varrho})<\gamma^2-\lambda\sup_{w\in \Phi^{-1}((-\infty,\gamma^2])}\Psi(w).
\end{equation}
\end{proposition}

\begin{proof}
The estimate \eqref{phigamma}  follows at once from \eqref{tiau2} and \eqref{norma3}.

From \eqref{Segno} we infer
\begin{eqnarray}\label{l}
\int_\Omega \beta(x)F(\Tr(w_{\tau}^{\varrho})(x))\;dx\geq \beta_0\omega_n\frac{\tau^n}{2^n}F(\varrho),
\end{eqnarray}
which, together with \eqref{norma3}, gives
\begin{equation}\label{condition}
\frac{\Psi(w_{\tau}^{\varrho})}{\Phi(w_{\tau}^{\varrho})}\geq \beta_0\frac{\omega_n}{h^{(n)}_\Omega}\frac{\tau^n}{2^n}\frac{F(\varrho)}{\varrho^2}.
\end{equation}

Since $\mu_1<\lambda<\mu_2$, we obtain
$$
\chi(\gamma^2)\leq \frac{1}{\mu_2}<\frac{1}{\lambda}<\frac{1}{\mu_1}\leq \frac{\Psi(w_{\tau}^{\varrho})}{\Phi(w_{\tau}^{\varrho})}
$$
and therefore
\begin{equation}\label{norma22finale}
\begin{aligned}
   \frac{\Psi(w_{\tau}^{\varrho})-\displaystyle\sup_{w\in \Phi^{-1}((-\infty,\gamma^2])}\Psi(w)}{\Phi(w_{\tau}^{\varrho})-\gamma^2}&\geq
   \frac{\Psi(w_{\tau}^{\varrho})-\gamma^2\displaystyle\frac{\Psi(w_\tau^{\varrho})}{\Phi(w_\tau^{\varrho})}}{\Phi(w_{\tau}^{\varrho})-\gamma^2}   \\
   &=\frac{\Psi(w_{\tau}^{\varrho})}{\Phi(w_{\tau}^{\varrho})} \\
   &\geq \frac{1}{\mu_1} >\frac{1}{\lambda}.
\end{aligned}
\end{equation}
This implies inequality \eqref{phigamma2} and the proof is thus completed.
\end{proof}

Associated with $\gamma$ let us consider, for every $w\in X^{1/2}_0(\mathcal{C}_\Omega)$, the following truncated functional
\begin{equation*}\label{deftruncfunc}
\mathcal J_{\lambda}^{(\gamma)}(w):=
\left\{
\begin{array}{ll}
\gamma^2-\lambda \Psi(w) & \mbox{ if $w\in \Phi^{-1}((-\infty,\gamma^2])$} \\
\mathcal J_{\lambda}(w) & \mbox{ if $w\notin \Phi^{-1}((-\infty,\gamma^2])$.}
\end{array}
\right.
\end{equation*}
Fixing $\lambda>0$, since $\mathcal J_{\lambda}^{(\gamma)}$ is sequentially weakly lower semicontinuous and coercive on the Hilbert space $X^{1/2}_0(\mathcal{C}_\Omega)$, it will attain a global minimum $w_{2,\lambda}$, that is:
\begin{equation}\label{globminu}
\mathcal J_{\lambda}^{(\gamma)}(w_{2,\lambda})\leq \mathcal J_{\lambda}^{(\gamma)}(w),\quad \forall\, w\in X^{1/2}_0(\mathcal{C}_\Omega).
\end{equation}

The next result states the impossibility for $w_{2,\lambda}$ to lie in the ball $B(0,\sqrt{2}\gamma)$ when $\lambda\in(\mu_1,\mu_2)$.

\begin{theorem}\label{Prop3}
Assume $\lambda\in(\mu_1,\mu_2)$. Then
$$
w_{2,\lambda}\notin \Phi^{-1}((-\infty,\gamma^2])
$$
and $\mathcal J_{\lambda}'(w_{2,\lambda})=0$, i.e. $w_{2,\lambda}\in X^{1/2}_0(\mathcal{C}_\Omega)$ is a critical point of $\mathcal J_{\lambda}$.
\end{theorem}

\begin{proof}
Pick $\lambda\in(\mu_1,\mu_2)$ and, arguing by contradiction, assume that $$
w_{2,\lambda}\in \Phi^{-1}((-\infty,\gamma^2]).
$$
In view of \eqref{globminu} it follows that
\begin{equation}\label{globminu11}
\gamma^2-\lambda \Psi(w_{2,\lambda})\leq \mathcal J_{\lambda}(w_\tau^{\varrho}),
\end{equation}
while due to Proposition \ref{Prop2} one has
\begin{equation}\label{globminu12}
\mathcal J_{\lambda}(w_\tau^{\varrho}) < \gamma^2-\lambda \Psi(w_{2,\lambda}).
\end{equation}
Collecting inequalities \eqref{globminu11} and \eqref{globminu12} we get the desired contradiction.
\end{proof}

We are now in position to prove our main result.
\begin{theorem}\label{generalkernel0f}
Let $\Omega$ be an open bounded set of $\R^n$ $\mathopen{(}n\geq 2\mathclose{)}$ with Lipschitz boundary $\partial\Omega$, $\beta:\Omega\to\R$ an $L^\infty$-map satisfying \eqref{proprietabeta} and $f:\R\to\R$ a continuous function satifying \eqref{S} and \eqref{Segno}. Furthermore, assume that the following algebraic inequality holds
\begin{equation}\label{AI}
\frac{F(\varrho)}{\varrho^2}> \displaystyle a_1\frac{K_1}{\gamma}+a_2K_2\gamma^{q-2}
\end{equation}
for some $\varrho,\gamma >0$ satifying \eqref{tiau2}, in addition to
\begin{equation}\label{AII}
F(t)\leq b(1+|t|^l)
\end{equation}
for all $t\in\R$ and for some positive constants $b$ and $l<2$.

Then for each $\lambda\in (\mu_1,\mu_2)$, problem \eqref{problema} has at least three weak solutions  $u_{1,\lambda}, u_{2,\lambda}, u_{3,\lambda}\in L^{\infty}(\Omega) \cap H_0^{1/2}(\Omega)$.
\end{theorem}

\begin{proof}
It is easy to verify that condition \eqref{AI} forces $\mu_1 <\mu_2$. Then fixing $\lambda\in (\mu_1,\mu_2)$ and appealing to Theorems \ref{Prop1} and \ref{Prop3}, we can deduce the existence of the first two solutions, recalling that $u_{j,\lambda}=Tr(w_{j,\lambda})$, $j=1,2$. The nature of local minima of the functions $w_{1,\lambda}$, $w_{2,\lambda}$ permits us to apply Theorem \ref{PucciSerrin}. Let us show that $\mathcal J_{\lambda}$ satisfies $\textrm{(PS)}_{\mu}$ for $\mu\in\R$. Let $\{w_{j}\}_{j\in \N}\subset X^{1/2}_0(\mathcal{C}_\Omega)$ satisfy
$$
\mathcal J_{\lambda}(w_j)\to \mu \quad{\rm and}\quad \|\mathcal J'_{\lambda}(w_j)\|_{*}\to 0
$$
as $j\rightarrow +\infty$, where
$$
\|\mathcal J'_{\lambda}(w_j)\|_{*}:=\sup\Big\{\big|\langle\,\mathcal J_{\lambda}'(w_j),\varphi \rangle
\big|: \; \varphi\in X^{1/2}_0(\mathcal{C}_\Omega) \quad {\rm and}\quad \|\varphi\|_{X^{1/2}_0(\mathcal{C}_\Omega)}=1\Big\}.
$$

Since $l<2$, for every $w\in X^{1/2}_0(\mathcal{C}_\Omega)$ one has $|\Tr(w)|^{l}\in L^{2/l}(\Omega)$ and combined with H\"{o}lder's inequality this gives
$$
\int_{\Omega}|\Tr(w)(x)|^ldx\leq |\Omega|^{\frac{2-l}{2}}\|\Tr(w)\|_{L^2(\Omega)}^{l} \quad\forall\, w\in X^{1/2}_0(\mathcal{C}_\Omega)
$$
and, by \eqref{constants},
\begin{equation}\label{co}
\int_{\Omega}|\Tr(w)(x)|^ldx\leq c_2^l|\Omega|^{\frac{2-l}{2}}\|w\|^{l}_{X^{1/2}_0(\mathcal{C}_\Omega)} \quad \forall\, w\in X^{1/2}_0(\mathcal{C}_\Omega).
\end{equation}
In the light of inequalities \eqref{AII} and \eqref{co} we get
$$
\mathcal J_{\lambda}(w)\geq \frac{1}{2}\|w\|^2_{X^{1/2}_0(\mathcal{C}_\Omega)}-\lambda bc_2^l\|\beta\|_{\infty}|\Omega|^{\frac{2-l}{2}}\|w\|^{l}_{X^{1/2}_0(\mathcal{C}_\Omega)}-\lambda b\|\beta\|_{\infty}|\Omega| \quad\forall\, w\in X^{1/2}_0(\mathcal{C}_\Omega),
$$
and therefore $\mathcal J_{\lambda}$ is bounded from below and
$\mathcal J_{\lambda}(w)\rightarrow +\infty$ as $\|w\|_{X^{1/2}_0(\mathcal{C}_\Omega)}\rightarrow +\infty$. This allows us to deduce that the sequence $\{w_{j}\}_{j\in \N}$ is bounded in $X^{1/2}_0(\mathcal{C}_\Omega)$. Since $X^{1/2}_0(\mathcal{C}_\Omega)$ is reflexive we can extract a subsequence, for simplicity denoted again $\{w_{j}\}_{j\in \N}$, such that
$w_{j}\rightharpoonup w_\infty$ in $X^{1/2}_0(\mathcal{C}_\Omega)$, i.e.,
\begin{equation}\label{conv0}
\int_{\mathcal{C}_\Omega}\langle \nabla w_j,\nabla \varphi \rangle dxdy\to
\int_{\mathcal{C}_\Omega}\langle \nabla w_{\infty},\nabla \varphi \rangle dxdy
\end{equation}
as $j\to+\infty$ for any $\varphi\in X^{1/2}_0(\mathcal{C}_\Omega)$.

We will prove that $w_{j}\to w_\infty$ as $j\to+\infty$. Keeping \eqref{derivatefunz} in mind, one has
\begin{equation}\label{jj}
\langle \Phi'(w_{j}),w_{j}-w_\infty\rangle = \langle \mathcal J_{\lambda}'(w_j),w_j-w_\infty\rangle + \lambda\int_{\Omega}\beta(x)f(\Tr(w_j)(x))\Tr(w_j-w_\infty)(x)dx,
\end{equation}
where
$$
\langle \Phi'(w_{j}),w_{j}-w_\infty\rangle = \displaystyle\int_{\mathcal{C}_\Omega}|\nabla w_j(x,y)|^2\,dxdy
              - \int_{\mathcal{C}_\Omega}\langle \nabla w_j,\nabla w_\infty \rangle dxdy.
$$

Since $\|\mathcal J_{\lambda}'(w_j)\|_{*}\to 0$ and the sequence $\{w_j-w_\infty\}_{j\in \N}$ is bounded in $X^{1/2}_0(\mathcal{C}_\Omega)$, taking into account the fact that $|\langle\mathcal J_{\lambda}'(w_j),w_j-w_\infty\rangle|\leq\|\mathcal J_{\lambda}'(w_j)\|_{*}\|w_j-w_\infty\|_{X^{1/2}_0(\mathcal{C}_\Omega)}$, one has
\begin{eqnarray}\label{j2}
\langle \mathcal J_{\lambda}'(w_j),w_j-w_\infty\rangle\to 0
\end{eqnarray}
as $j\to+\infty$. Furthermore, by \eqref{S} and H\"{o}lder's inequality one has
\begin{align*}
\int_{\Omega}\beta(x)|f(\Tr(w_j)(x))||\Tr(w_j-w_\infty)(x)|dx &\leq a_1 \left\|\beta\right\|_\infty\int_\Omega |\Tr(w_j-w_\infty)(x)|dx \\
& \quad + a_2\left\|\beta\right\|_\infty \int_\Omega |\Tr(w_j)(x)|^{q-1}|\Tr(w_j-w_\infty)(x)|dx \\
&\leq a_1 \left\|\beta\right\|_\infty \left\|\Tr(w_j-w_\infty)\right\|_1\\
& \quad +
 a_2\left\|\beta\right\|_\infty \left\|\Tr(w_j)\right\|_q^{q-1}\left\|\Tr(w_j-w_\infty)\right\|_q.
\end{align*}
Since the embeddings $\Tr(X^{1/2}_0(\mathcal{C}_\Omega))\hookrightarrow L^1(\Omega)$, $\Tr(X^{1/2}_0(\mathcal{C}_\Omega))\hookrightarrow L^q(\Omega)$ are compact, we obtain
\indent \begin{eqnarray}\label{j3}
\int_{\Omega}\beta(x)|f(\Tr(w_j)(x))||\Tr(w_j-w_\infty)(x)|dx\to 0
\end{eqnarray}
as $j\rightarrow +\infty$.

Relations \eqref{j2} and \eqref{j3} force
\begin{eqnarray}\label{fin}
\langle \Phi'(w_{j}),w_{j}-w_\infty\rangle
\rightarrow 0
\end{eqnarray}
as $j\rightarrow +\infty$ and hence
\begin{equation}\label{fin2}
\displaystyle\int_{\mathcal{C}_\Omega}|\nabla w_j(x,y)|^2\,dxdy
- \int_{\mathcal{C}_\Omega}\langle \nabla w_j,\nabla w_\infty \rangle dxdy\rightarrow 0
\end{equation}
as $j\rightarrow +\infty$. Thus it follows by \eqref{conv0} and \eqref{fin2} that
$$
\lim_{j\rightarrow +\infty}\displaystyle\int_{\mathcal{C}_\Omega}|\nabla w_j(x,y)|^2\,dxdy = \displaystyle\int_{\mathcal{C}_\Omega}|\nabla w_\infty(x,y)|^2\,dxdy,
$$
as desired, and the third solution to \eqref{problema} is obtained as well.
Finally, due to \eqref{S} and the essential boundedness of $\beta$, by \cite[Theorem 5.2]{cabretan} we get $u_{i,\lambda}:=\Tr(w_{i,\lambda})\in L^{\infty}(\Omega)$ for $i\in\{1,2,3\}$.
\end{proof}

As a corollary of the previous result, we can deduce Theorem \ref{boccia} stated in the introduction.\\

\noindent\emph{Proof of Theorem \ref{boccia}}. Let us consider the non-negative continuous function $f^{\star}:\R\rightarrow\R$ defined by
\[
f^{\star}(t):= \left\{
\begin{array}{ll}
\displaystyle f(t) & \mbox{ if\, $0<t\leq \zeta$} \\\\
\displaystyle 0 & \mbox{ otherwise,}
\end{array}
\right.
\]
and apply Theorem \ref{generalkernel0f} to the problem
\begin{equation}\label{problema244}
\left\{
\begin{array}{ll}
A_{1/2}u= \lambda \beta(x)f^{\star}(u) & \mbox{\rm in } \Gamma^0_r\\
u> 0 & \mbox{\rm on } \Gamma^0_r\\
u=0 & \mbox{\rm on } \partial\Gamma^0_r.
\end{array}\right.
\end{equation}
Now, fixing
$$
\lambda>\frac{2^n h^{(n)}_\Omega}{\omega_n\tau^n\beta_0}\displaystyle\inf_{0<\varrho\leq \zeta}\frac{\varrho^2}{F(\varrho)}
$$
there exists $\bar{\varrho}>0$ such that
$$
\lambda>\frac{2^n h^{(n)}_\Omega}{\omega_n\tau^n\beta_0}\displaystyle\frac{\bar\varrho^2}{F(\bar\varrho)}.
$$
If we take
$$
\gamma<\min\left\{\sqrt{g^{(n)}_\Omega}\bar{\varrho},\left(\frac{2^n h^{(n)}_{\Omega}}{a_2\omega_n\tau^n\beta_0K_2\lambda}\right)^{1/(m-1)}\right\},
$$
it is easily seen that all the assumptions of Theorem \ref{generalkernel0f} are satisfied. Hence there exist at least two weak solutions of problem \eqref{problema244} and, in turn, of problem \eqref{problema24}.
\qed

\begin{remark}
We point out that the operator $A_{1/2}$ we considered here is not to be
confused with the integro-differential operator defined, up to a constant, by
$$
(-\Delta)^{1/2} u(x):=
-\int_{\mathbb{R}^{n}}\frac{u(x+y)+u(x-y)-2u(x)}{|y|^{n+1}}\,dy, \quad\, \forall\, x\in \R^n.
$$
In fact, Servadei \& Valdinoci in \cite{svd} showed that these two operators, though often denoted the same way, are really different, with eigenvalues and eigenfunctions behaving quite differently (see also Musina \& Nazarov \cite{MuNa}).
\end{remark}

\begin{remark}
The techniques adopted in this paper are still valid if we consider the more general problem:
\begin{equation}\tag{$P_{\alpha,\lambda}$}\label{problemaalfa}
	\left\{
	\begin{array}{ll}
		A_{\alpha/2}u=\lambda \beta(x)f(u) & \mbox{ in } \Omega\\
		u=0 & \mbox{ on } \partial\Omega,
	\end{array}\right.
\end{equation}
with $\alpha\in(0,2)$. The operator $A_{\alpha/2}$ is defined by
$$
A_{\alpha/2}u(x):=-\kappa_\alpha\lim_{y\rightarrow 0^+}y^{1-\alpha}\frac{\partial \E(u)(x,y)}{\partial y}\quad\forall\, x\in \Omega,
$$
where $u$ belongs to the space
$$
H_0^{\alpha/2}(\Omega):=\left\{u\in L^2(\Omega): u=\displaystyle\sum_{j=1}^{\infty}a_j\varphi_j \mbox{  and } \displaystyle \left\| u\right\|_{H_0^{\alpha/2}(\Omega)}:= \left( \sum_{j=1}^{\infty} a_j^2\lambda_j^{\alpha/2}\right)^{1/2}  <+\infty\right\},
$$
and its $\alpha$-harmonic extension $\E_\alpha(u)$ to the cylinder $C_\Omega$ is the unique solution of the local problem
\begin{equation}\label{estensionealpha}
	\left\{
	\begin{array}{ll}
		\div(y^{1-\alpha}\nabla \E_\alpha(u))=  0  &\mbox{ in } \mathcal{C}_\Omega\\
		\displaystyle{\E_\alpha(u)=0}  &\mbox{ on } \partial_L\mathcal{C}_\Omega\\
		\Tr(\E_\alpha(u))=u  &\mbox{ on } \Omega.
	\end{array}\right.
\end{equation}
\end{remark}
Such an extension lies in the space
$$
X^{\alpha/2}_0(\mathcal{C}_\Omega):=\left\{w\in L^2(\mathcal{C}_\Omega): w=0\,\,{\rm on }\,\, \partial_L \mathcal{C}_\Omega,\,\, \left\| w\right\|_{X^{\alpha/2}_0(\mathcal{C}_\Omega)}<+\infty\right\},
$$
where
$$
\left\| w\right\|_{X^{\alpha/2}_0(\mathcal{C}_\Omega)}:= \left( \kappa_\alpha\int_{\mathcal{C}_\Omega}y^{1-\alpha}|\nabla w(x,y)|^2dxdy\right) ^{1/2}
$$
and 
$$
\kappa_\alpha:=\frac{\displaystyle\Gamma\left(\frac{\alpha}{2}\right)}{2^{1-\alpha}\displaystyle\Gamma\left(1-\frac{\alpha}{2}\right)}
$$
is a normalization constant which makes the operator  $\E_\alpha:H_0^{1/2}(\Omega)\rightarrow X^{1/2}_0(\mathcal{C}_\Omega)$ an isometry. As a class of test functions in $X^{\alpha/2}_0(\mathcal{C}_\Omega)$ necessary for our approach one can choose
$$
w_{\tau}^{\varrho}(x,y):=\frac{1-e^{-\alpha y}}{\alpha y}\omega_\tau^{\varrho}(x),\quad\, \forall (x,y)\in \mathcal{C}_\Omega,
$$
$\omega_\tau^{\varrho}$ being defined by \eqref{deftruncfunc}.

\bigskip

\noindent {\bf Acknowledgements.} The paper has been carried out under the auspices of the INdAM - GNAMPA Project 2016 titled: {\it Problemi variazionali su variet\`a Riemanniane e gruppi di Carnot} and the Slovenian Research Agency grants P1-0292, J1-7025, J1-6721 and J1-5435.

\end{document}